\documentclass[a4paper,12pt]{article}
\usepackage[utf8]{inputenc}
\usepackage{graphicx}
\usepackage{nicefrac}
\usepackage{amsmath}
\usepackage{amssymb}
\usepackage{amsthm}
\usepackage{url}
\usepackage{amscd}
\usepackage{tikz}

\usepackage{multicol}
\usepackage{array}
\newtheorem{TheoM}{Theorem}

\newtheorem{TheoP}{Theorem}

\newtheorem*{defi-intro}{Definition}
\newtheorem*{Theo}{Theorem}

\newtheorem{Theorem}{Theorem}[section]

\newtheorem{Proposition}[Theorem]{Proposition}
\newtheorem{Lemma}[Theorem]{Lemma}
\newtheorem{defi}[Theorem]{Definition}

\theoremstyle{definition}
\newtheorem{Example}[Theorem]{Example}

\newcommand{\TM}{\operatorname{Gr}\mathcal{M}}

\theoremstyle{remark}
\newtheorem{Remark}[Theorem]{Remark}

\title{Face module for realizable $\mathbb{Z}$-matroids\footnote{Accepted for publication in Contributions to Discrete Mathematics.}}
\author{Ivan Martino}

\begin{document}

\maketitle

\begin{abstract}
In this work, we define the face module for a realizable matroid over $\mathbb{Z}$. 
Its Hilbert series is, indeed, the expected specialization of the Grothendieck--Tutte polynomial defined by Fink and Moci in \cite{FinkMoci}.
\end{abstract}

A {\em matroid} $M$ is a simplicial complex $\mathcal{I}$ on a {\em ground set} $[n]=\{1,\dots,n\}$, such that
\begin{equation*}%\label{eq-ind-set-exchange}
 	A, B \in \mathcal{I}, \#A>\#B \Rightarrow \exists a\in A\setminus B: B\cup \{a\}\in \mathcal{I}.
\end{equation*}
%$$if A, B \in \mathcal{I}$ and $\#A>\#B$, then $\exists a\in A\setminus B: B\cup \{a\}\in \mathcal{I}.$$ 
%
The latter is called {\em independent set exchange property} and $\mathcal{I}$ is often called {\em independent sets} family. % and the facet of $\mathcal{I}$ are called {\em basis} of the matroid.
%
%\noindent
Matroids encapsulate the combinatorics that underline the arrangements of hyperplanes in affine or projective space. %; surely many other application can be mentioned like optimizations and 

\noindent
There are two classical objects one associates to a matroid: the Stanley-Reisner ring $\mathbf{k}[M]$, that is the face ring of $\mathcal{I}$, and the Tutte polynomial $T_{M}$. %:
%\begin{equation*}
%	T_{M}(x,y)=\sum_{A\subseteq [n]} (x-1)^{r-d(A)}(y-1)^{\#A - d(A)}.
%\end{equation*}
%
They are related by the following result:

\begin{Theo}
	Let $M$ be a matroid of rank $r$ with ground set $[n]$ and call $M^*$ its dual matroid. Then:
    \[
        \operatorname{Hilb}(\mathbf{k}[M],t)=\frac{t^{r}}{(1-t)^{r}} T_{M^*}(1, \nicefrac{1}{t}).
    \]
    where $\operatorname{Hilb}(\mathbf{k}[M],t)$ is the Hilbert series of $\mathbf{k}[M]$.
\end{Theo}

Fink and Moci \cite{FinkMoci} generalize the concept of matroid to a larger setting: a matroid $\mathcal{M}$ over a commutative ring $R$ on the ground set $[n]$ is an assignment of an $R$-module $\mathcal{M}(A)$ for every subset $A$ of $[n]$. This assignment has to respect a certain local patching condition.
%
%\noindent
One of the reason behind this generalization is to deal with arrangements of hypersurfaces. For the goal of this paper, it is worth to recall that a realizable matroid over $\mathbb{Z}$ relates to a generalized toric arrangement \cite{Moci-A-tutte,Dadderio-Moci-Arith}. 
%
%Toric arrangements have been vastly studied 

In this paper, we introduce a candidate for the role of the independent set complex for realizable matroids over $\mathbb{Z}$, that we call {\em partially order set (poset) of torsions} of $\mathcal{M}$. 
This new poset provides a combinatorial tool for (generalized) toric arrangements to compute their integral cohomology \cite{Delucchi-Callegaro-toric}, to construct their wonderful models \cite{Moci-Wonderful-Model} and their projective wonderful models \cite{DeC-Gaiffi-Projective-Wonderful}.
Moreover, this paves the way to show one of Rota's cryptomorphisms for matroids over $\mathbb{Z}$.

%\noindent
Given a matroid $\mathcal{M}$ over $\mathbb{Z}$, $\mathcal{M}(A)$ is an abelian group and so 
$$\mathcal{M}(A)=\mathbb{Z}^{d(A)}\times G_A,$$
where $G_A$ is a torsion group and its cardinality, $\#G_A$, is often referred as the multiplicity of $A$, $m(A)$. Call $d=d(\emptyset)$, the dimension of the matroid $\mathcal{M}$ and $C_A=\operatorname{Hom}(G_A,\mathbb{C}^*)$, the dual group of $G_A$. 

\begin{defi-intro}
	We denote by $\TM$ the set of torsions of $\mathcal{M}$. This is the set of all pairs $(A, l)$ with $d-d(A)=\#A$ and $l\in C_A$.
\end{defi-intro}

\noindent
First, observe that $(\emptyset, e)$ always belongs to $\TM$.
We are going to give a partial order to $\TM$ by defining certain covering relations inspired by the poset of layers of a toric arrangement \cite{Moci-A-tutte,Dadderio-Moci-Arith}.
Similar ad hoc constructions appear also in \cite{Delucchi-Callegaro-toric}.
This order depends on the realization of the matroids $\mathcal{M}$. Indeed, every realization provides a surjective map $\pi:C_{A\cup b}\twoheadrightarrow C_A$. (See Section \ref{sec-coset-of-poset} for further details.)

\begin{defi-intro}
	Let $\mathcal{M}$ be a realized $\mathbb{Z}$-matroid.
	Let $(A\cup \{b\}, h)$ and $(A, l)$ be two elements of $\TM$.
	We say that $(A\cup \{b\}, h)$ covers $(A, l)$ if and only if $\pi(h) = l$.
\end{defi-intro}

This poset is not a simplicial complex, but it is the union of identical simplicial posets. 

\begin{TheoM}\label{thm-simplicial-poset-intro}
	%For every representable matroid $\mathcal{M}$ over $\mathbb{Z}$ with $\mathcal{M}(\emptyset)=\mathbb{Z}^{d}$, $\TM$ is a simplicial poset.
    If $\mathcal{M}$ is a realized $\mathbb{Z}$-matroid, then $\TM$ is a disjoint union of $m(\emptyset)$ simplicial posets isomorphic to the link of $(\emptyset, e)$ in the poset $\TM$.
\end{TheoM}

\noindent
As a byproduct, one can reproduce many of the results of Section 5 and 6 of \cite{Moci-A-tutte}, but with the payback of losing part of the geometrical intuition.

From the poset $\TM$ we define a {\em face module} $\mathbf{k}[\mathcal{M}]$ associated to $\mathcal{M}$. 
For this, we use Stanley's construction \cite{Stanley-f-vector} of the face ring for simplicial posets.
As proof that $\TM$ is the right combinatorial object to study, we also show that the Hilbert series of its face \emph{module} is the specialization of the Grothendieck--Tutte polynomial, as in the classical case:
\begin{TheoM}\label{thm-main-result}
	If $\mathcal{M}$ ia a realizable $\mathbb{Z}$-matroid of rank $r$, then
    \[
    \operatorname{Hilb}(\mathbf{k}[\mathcal{M}],t)=\frac{t^{r}}{(1-t)^{r}} T_{\mathcal{M}^*}(1, \nicefrac{1}{t}).
    \]
\end{TheoM}

\noindent
The Grothendieck--Tutte polynomial for matroid over a ring has been defined by Fink and Moci \cite{FinkMoci} as a function of $\mathcal{M}$ in a certain  Grothendieck ring of matroids, but, in our setting, $T_{\mathcal{M}}$ is more concretely the arithmetic Tutte polynomial, see \cite{Dadderio-Moci-Arith}. Precisely,
\begin{equation*}
	T_{\mathcal{M}}(x,y)=\sum_{A\subseteq [n]} m(A) (x-1)^{r-\operatorname{cork}(A)}(y-1)^{\#A - \operatorname{cork}(A)},
\end{equation*}
where $\operatorname{cork}(A)=d(\emptyset)-d(A)$.

\noindent
It is worth to mention that it is not clear if the poset of torsion is uniquely defined for non realizable matroids. 
On the other hands, the proof of Theorem \ref{thm-main-result} holds easily for any simplicial poset with the correct $\mathbf{f}$-vector.
Therefore we conjecture that Theorem \ref{thm-main-result} holds for every $\mathbb{Z}$-matroid.
See more about this in Remark \ref{remark-conjecture}.

%\begin{Conjecture}
%	Let $\mathcal{M}$ be a $\mathbb{Z}$-matroid of rank $r$ with ground set $[n]$. 
%	%
%	%Assume $\mathcal{M}(\emptyset)$ has no torsion.
%	%
%	Let $\TM$ be the simplicial poset of torsions of $\mathcal{M}$ and denote its face ring by $\mathbf{k}[\mathcal{M}]$.
%	%
%	Then 
%	\[
%	\operatorname{Hilb}(\mathbf{k}[\mathcal{M}],t)=\frac{t^{r}}{(1-t)^{r}} T_{\mathcal{M^*}}(1, \nicefrac{1}{t}).
%	\]
%\end{Conjecture}
%
%The main obstacle to show this result is hidden in the definition of the simplicial poset of torsion $\TM$. 

%\vspace{0.1cm}
%\medskip%{0.1cm}
The paper is organized as follows: in Section \ref{sec-basic-notion} we recall all the basic notions needed for a full comprehension of the results. In Section \ref{sec-coset-of-poset} we define the poset of torsions and in Section \ref{sec-simplicial-poset} we prove Theorem \ref{thm-simplicial-poset-intro}. Finally, in Section \ref{sec-main-result} we show Theorem \ref{thm-main-result}.  
%\vspace{0.1cm}
%\medskip%{0.1cm}
\paragraph{Acknowledgements.} 
The author thanks an anonymous referee for several fruitful observations and comments.
A naive version of this poset was discussed with Emanuele Delucchi, Matthias Lenz and Luca Moci. 
The author thanks Matthias Lenz and Matthew Stamps for several inspiring discussions. The author is also grateful to Emanuele Delucchi for the best introduction to the matroid world, and Alex Fink and Luca Moci for the email exchanges on their result on matroids over a ring.
Finally, the author thanks Emanuele Delucchi and Alex Fink for pointing out a mistake in the previous version of the paper.

The author has been partially supported by the \emph{Swiss National Science Foundation Professorship} grant PP00P2\_150552/1. 
Moreover, has been partially supported by the \emph{Zelevinsky Research Instructor Fund}.

\noindent
Currently, the author is supported by the \emph{Knut and Alice Wallenberg Fundation} and by the \emph{Royal Swedish Academy of Science}.

\section{Basic notions}\label{sec-basic-notion}

\subsection{Simplicial posets}
%\paragraph{Simplicial posets.}
Let $(P, <)$ be a finite partially ordered set (poset).
A poset with a unique initial element, denoted by $\hat{0}$, is said to be {\em simplicial} if for each $\sigma\in P$ the segment $[\hat{0}, \sigma]=\{x\in P: \hat{0}\leq x \leq \sigma\}$ is a boolean lattice.
We say that the \emph{rank} of $[\hat{0}, \sigma]$ is the length of its maximal chain; therefore $(P, <)$
has a natural rank function $\operatorname{rk}$ induced by the rank of the segments $[\hat{0}, \sigma]$.
We denote by $r$ the rank of $P$, the maximal rank among all its segments.%, and by $d=r-1$ the dimension of $P$, that is the dimension of the order complex of $P$.

\noindent
For any $\sigma$ and $\tau$ in $P$, $\sigma\wedge\tau$ is the set of their greatest lower bounds ({\em meets}) and $\sigma\vee\tau$ is the set of their least common upper bounds ({\em joins}).
For a simplicial poset, $\sigma\wedge\tau$ is a singleton and by abuse of notation we identify the $\sigma\wedge\tau$ with the unique greatest lower bound of $\sigma$ and $\tau$.

\begin{Example}\label{ex:basic-simplicial-poset}
    Consider the set given by $P_1=\{\hat{0}, a, b, 1, 2\}$, where every number is greater or equal to every letter and every element is greater or equal to $\hat{0}$, see Figure \ref{fig:poset-example}.a).
    This is a simplicial poset. It is not the face poset of any simplicial complex, but it is the face poset of a digon, a CW-complex shown in Figure \ref{fig:poset-example}.b). Its order complex is a triangulation of the one dimensional sphere: see Figure \ref{fig:poset-example}.c). 
    We compute few examples of meets and joins that are useful in future computations: $1\wedge a = \{a\}$, $a \wedge b = \{\hat{0}\}$, $a \vee b = \{1, 2\}$ and $1 \vee 2 = \emptyset$.
\end{Example}

\begin{figure}[htb]
\begin{center}
\begin{tabular}{m{3cm}m{2cm}m{3cm}m{3cm}}
        \begin{tikzpicture}[scale=0.6, align=center]
            \node (1) at (0,0) {$1$};
            \node (2) at (3,0) {$2$};
            \node (a) at (0, -2) {$a$};
            \node (b) at (3, -2) {$b$};
            \node (0) at (1.5, -4) {$\hat{0}$};
            \draw (1) -- (a);
            \draw (2) -- (a);
            \draw (1) -- (b);
            \draw (2) -- (b);
            \draw (a) -- (0);
            \draw (b) -- (0);
        \end{tikzpicture}
&
\begin{tikzpicture}[scale=.6, align=center]
            \node (a) at (0,0) {$a$};
            \node (b) at (0, -4) {$b$};
            \draw [thick, shorten <=-2pt, shorten >=-2pt] (a) to [out=-90, in=+60] (b);
            \draw [thick, shorten <=-2pt, shorten >=-2pt] (a) to [out=-90, in=+120] (b);
        \end{tikzpicture}
&        \begin{tikzpicture}[scale=.6, align=center]
            \node (1) at (0,0) {$1$};
            \node (2) at (3, -2) {$2$};
            \node (a) at (0, -2) {$a$};
            \node (b) at (3,0) {$b$};
            \draw (1) -- (a);
            \draw (2) -- (a);
            \draw (1) -- (b);
            \draw (2) -- (b);
        \end{tikzpicture}
&		\begin{tikzpicture}[scale=0.6, align=center]
            \node (1) at (0,0) {$1$};

 			\node (a) at (-3, -2) {$a$};
            \node (b) at (0, -2) {$b$};
 			\node (c) at (3, -2) {$c$};
 
            \node (0) at (0, -4) {$\hat{0}$};

			\draw (1) -- (a);
            \draw (1) -- (b);
            \draw (1) -- (c);
  
            \draw (a) -- (0);
            \draw (b) -- (0);
            \draw (c) -- (0);
        \end{tikzpicture}\\
% The poset in Example \ref{ex:basic-simplicial-poset}
% &
% The CW-complex with face poset $P$
% &
% The order complex of $P$
 \qquad \, a) & \hspace{0.3cm} b) & \qquad \, c) &  \hspace{1.75cm} d)
\end{tabular}
\end{center}
\caption{a) The poset $P_1$ in Example \ref{ex:basic-simplicial-poset}; b) The CW-complex with face poset $P_1$; c) The order complex of $P_1$; d) The poset $P_2$ in Example \ref{ex:basic-non-simplicial-poset-simple}.}
   \label{fig:poset-example}
\end{figure}
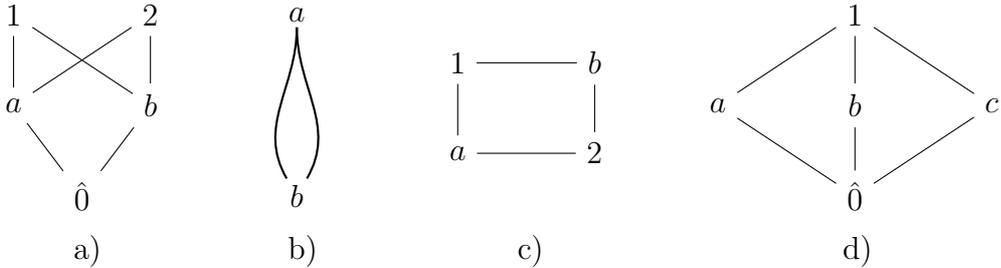

% \begin{figure}
% %     \begin{tabular}{cc}
% \begin{center}
%         \begin{tikzpicture}[scale=.6, align=center]
%             \node (1) at (0,0) {$1$};
%             \node (2) at (3,0) {$2$};
%             \node (a) at (0, -2) {$a$};
%             \node (b) at (3, -2) {$b$};
%             \node (0) at (1.5, -4) {$\hat{0}$};
%             \draw (1) -- (a);
%             \draw (2) -- (a);
%             \draw (1) -- (b);
%             \draw (2) -- (b);
%             \draw (a) -- (0);
%             \draw (b) -- (0);
%         \end{tikzpicture}
%       \qquad \qquad
% %         & %
%         \begin{tikzpicture}[scale=.6, align=center]
%             \node (a) at (0,0) {$a$};
%             \node (b) at (0, -4) {$b$};
%             \draw [thick, shorten <=-2pt, shorten >=-2pt] (a) to [out=-90, in=+60] (b);
%             \draw [thick, shorten <=-2pt, shorten >=-2pt] (a) to [out=-90, in=+120] (b);
%         \end{tikzpicture}
% \end{center}
%         \caption{The poset in Example \ref{ex:basic-simplicial-poset}.%
%         %
%          \qquad The CW-complex with face poset $S$.}
%         %\label{fig:CW-complex-poset-example}
% %     \end{tabular}
%    \label{fig:poset-example}
%       \end{figure}

\begin{Example}\label{ex:basic-non-simplicial-poset-simple}
    Consider the set $P_2=\{\hat{0}, a, b, c, 1\}$ with the same order law given for $P_1$, see Figure \ref{fig:poset-example}.d).
    This is not a simplicial poset, because $[\hat{0}, 1]$ is not boolean.
    
% i1 : R=ZZ/101[x,y,z,w]

% o1 = R

% o1 : PolynomialRing

% i2 : I=ideal(y*z, y*w,z*w)

% o2 = ideal (y*z, y*w, z*w)

% o2 : Ideal of R

% i3 : hilbertSeries I

%            2     3
%      1 - 3T  + 2T
% o3 = -------------
%                4
%         (1 - T)

% o3 : Expression of class Divide

% i4 : 
\end{Example}

\subsection{Face ring}
%\paragraph{Face ring.}
Given a field $\mathbf{k}$, we set the polynomial ring $R_P=\mathbf{k}[x_{\sigma}: \sigma \in P]$ where $x_{\sigma}$ has degree $\operatorname{rk}{\sigma}$. In this recap section we are going to follow the notation in \cite{Stanley-f-vector}.

\begin{defi}
  The {\em face ideal} of a simplicial poset $P$ is the ideal of $R_P$ defined as
  \[
    I_{P}=\left(x_{\hat{0}}-1,\, x_{\sigma}x_{\tau} - x_{\sigma\wedge\tau}\left(\sum_{\gamma\in\sigma\vee\tau} x_{\gamma} \right) \mbox{ for any } \sigma, \tau \in P \right).
  \]
  As a notation, the sum $\sum_{\gamma\in\sigma\vee\tau} x_{\gamma}$ is zero if $\sigma\vee\tau=\emptyset$.
  Moreover, the {\em face ring} of a simplicial poset $P$ is the quotient 
  \[
    \mathbf{k}[P]=\frac{\mathbf{k}[x_{\sigma}: \sigma \in P]}{I_P}.
  \]
\end{defi}

\begin{Example}\label{ex:basic-simplicial-poset-ideal-computation}
    Consider the simplicial poset $P_1$ defined in Example \ref{ex:basic-simplicial-poset}, $R_{P_1}$ is the polynomial ring $\mathbf{k}[x_{\hat{0}}, x_a, x_b, x_1, x_2]$.
    
    The ideal $I_{P_1}=\left\langle x_{\hat{0}}-1, x_a x_b-(x_1+x_2), x_1x_2\right\rangle$. Indeed, the basis generators are obtained by substituting respectively $(a,b)$ and $(1,2)$ to the pair the $(\sigma, \tau)$ into $x_{\sigma}x_{\tau} - x_{\sigma\wedge\tau}\left(\sum_{\gamma\in\sigma\vee\tau} x_{\gamma} \right)$. 
    For all other values of $(\sigma, \tau)$, the previous relation is trivial.
    %
    %\noindent
    The face ring of $P_1$ is therefore defined as the quotient
    \[
        \mathbf{k}[P_1]=\frac{\mathbf{k}[x_{\hat{0}}, x_a, x_b, x_1, x_2]}{\left(x_{\hat{0}}-1, x_a x_b-(x_1+x_2), x_1x_2\right)}.
    \]
\end{Example}

This definition generalizes the {\em Stanley--Reisner ring} of a simplicial complex.
Given an abstract simplicial complex $\Delta$ on $n$ vertices its {\em Stanley--Reisner ring} $\mathbf{k}[\Delta]$ the following quotient ring $\mathbf{k}[\Delta]=\nicefrac{\mathbf{k}[x_1,\dots, x_n]}{I_{\Delta}}$, where %$I_{\Delta}$ is  the ideal generated by the square-free monomials corresponding to the non-faces of $\Delta$, i.e. 
$I_{\Delta}=\left\langle x_{i_1}\ldots x_{i_r}: \{i_1,\dots,i_r\}\notin\Delta \right\rangle$. 
%
%\noindent
One can easily check that the two face ring definitions coincide in the case of an abstract simplicial complex.

\begin{Example}\label{ex:basic-simplicial-Fig2-a}
	Consider the simplicial poset $P_2$ defined in Figure \ref{fig:poset-coset-example}.a).
    This is the face poset of a simplicial complex. Precisely, the path graph on $3$ elements.
    Its face ring, $\mathbf{k}[P_2]$, is isomorphic to $\nicefrac{\mathbf{k}[x, y, z]}{(xz)}$.
%     \[
%         \mathbf{k}[P_2]=\frac{\mathbf{k}[x, y, z]}{(xz)}.
%     \]
\end{Example}

From now on, we assume that $P$ is a simplicial poset, and that $\Delta$ is an abstract simplicial complex and $r$ denotes their ranks.

% \[
%     I_{\Delta}=\left\langle x_{i_1}\ldots x_{i_r}: \{i_1,\dots,i_r\}\notin\Delta \right\rangle
% \]
% and 
% \[    
%     \mathbf{k}[\Delta]=\nicefrac{k[x_e: e\in E]}{I_{\Delta}}. 
% \]

%$[n]=\{1,\dots, n\}$

\subsection{The Hilbert series of the face ring}
%\paragraph{Hilbert series of the face ring.}
%Setting the degree of the variable $x_\sigma$ equal the rank of $\sigma$, $\mathbf{k}[P]$ is graded rings. 
%
%For these, one can define a dimension counting series called {\em Hilbert series}.

%\begin{defi}
Let $N$ be a finitely generated $\mathbb{N}$-graded $A$-module where $A$ is a finitely generated $\mathbb{N}$-graded commutative algebra over $\mathbf{k}$. %, i.e. $A=\oplus_{i\geq 0} A_i$ where 
Denote by $N_i$ the homogeneous part of degree $i$.
The Hilbert series of $N$ is the following generating function:
\[
\operatorname{Hilb}(N, t)=\sum_{i\geq 0} \operatorname{dim}_{\mathbf{k}}(N_i)t^i,
\]
where $\operatorname{dim}_{\mathbf{k}}(N_i)$ is the dimension of $N_i$ as a $\mathbf{k}$-vector space.
We consider $\operatorname{Hilb}(A, t)$ as the Hilbert series of $A$ seen as a module over its self.
%
%Let $A$ be a finitely generated $\mathbb{N}$-graded commutative algebra over $\mathbf{k}$ %, i.e. $A=\oplus_{i\geq 0} A_i$ where 
%and denote by $A_i$ the homogeneous part of degree $i$.
%  %
%  The Hilbert series of $A$ is the following generating function:
%  \[
%    \operatorname{Hilb}(A, t)=\sum_{i\geq 0} \operatorname{dim}_{\mathbf{k}}(A_i)t^i,
%  \]
%  where $\operatorname{dim}_{\mathbf{k}}(A_i)$ is the dimension of $A_i$ as a $\mathbf{k}$-vector space.
%%\end{defi}

The ring $\mathbf{k}[P]$ is graded and its Hilbert series encodes many combinatorial objects, like the \emph{$\mathbf{f}$-vector} and the \emph{$\mathbf{h}$-vector}. 
Here, we recall shortly their definitions. 
The $\mathbf{f}$-vector, $\mathbf{f}(P)$, of a simplicial poset $P$ is the vector $(f_{-1}, f_0, \dots, f_{r-1})$ where $f_{i}$ is the number of elements of rank $i+1$ in $P$; by notation $f_{-1}=1$ counts the empty set as a dimension $-1$ object. 
The $\mathbf{h}$-vector of $P$ is the vector $\mathbf{h}(P)=(h_0, h_1, \dots, h_r)$ defined recursively from the $\mathbf{f}$-vector by using $\sum_{i=0}^{r} f_{i-1}(t-1)^{r-i}=\sum_{i=0}^r h_{i}t^{r-i}$.
%here, if necessary write here about h polynomial.

\begin{Example}\label{ex:basic-simplicial-poset-f-h-computation}
    We compute $\mathbf{f}(P_1)$ and $\mathbf{h}(P_1)$ for the simplicial poset in Example \ref{ex:basic-simplicial-poset}.
    Trivially, $\mathbf{f}(P_1)=(1,2,2)$. Expanding $\sum_{i=0}^2 f_{i-1}(t-1)^{2-i}$ one gets $t^2+1$ and therefore $\mathbf{h}(P_1)=(1,0,1)$.
\end{Example}

\begin{Example}\label{ex:basic-simplicial-Fig2-a-f-h-computation}
    Let us make similar computation for $P_2$ in Example \ref{ex:basic-simplicial-Fig2-a}. Clearly $\mathbf{f}(P_2)=(1,3,2)$ and by expanding $\sum_{i=0}^2 f_{i-1}(t-1)^{2-i}=t^2+t$ and therefore $\mathbf{h}(P_2)=(1,1,0)$.
\end{Example}

% \noindent
% For an abstract simplicial complex $\Delta$ on $n$ vertices, one can show that 
% \[
%   \operatorname{Hilb}(\mathbf{k}[\Delta], t) = \frac{\sum_{i=0}^r f_{i-1} t^i(1-t)^{n-i}}{(1-t)^n}= \frac{h_0 + h_1 t + \cdots + h_r t^{r}}{(1-t)^r}.
% \]
% % %
% % Moreover,
% % \[
% %   \operatorname{Hilb}(\mathbf{k}[\Delta], t) = \frac{h_0 + h_1 t + \cdots + h_r t^{r}}{(1-t)^r}.
% % \]
% %The proof of the latter result 
% %-Garsia \cite{Garsia-Combinatorial-methods}
% %-Kind and Kleinschmidt \cite{Kind-Kleinschmidt-Komplexe}
% %-Stanley \cite{Stanley-Combinatorics-commutative-algebra}

% \begin{Example}\label{ex:basic-simplicial-Fig2-a-Hilbert-series}
%     In the case of the face poset $P_2$, it is easy to verify what we have just stated.
%     %
%     Indeed, in Example \ref{ex:basic-simplicial-Fig2-a}, the face ring and in Example \ref{ex:basic-simplicial-poset-f-h-computation} we computed the face ring, the $\mathbf{f}$-vector and $\mathbf{h}$-vector.
%     %
%     %\noindent
%     It is trivial to observe the following:
%     \[
%         \operatorname{Hilb}\left(\frac{\mathbf{k}[x, y, z]}{(xz)}, t\right) = \frac{1-t^2}{(1-t)^3}=\frac{1+t}{(1-t)^2}.
%     \]
% \end{Example}

% \noindent
% Stanley extended the latter result to simplicial posets:
%
\noindent
As said, one can read the \emph{$\mathbf{f}$-vector} and the \emph{$\mathbf{h}$-vector} from the face ring $\mathbf{k}[P]$.

\begin{Theorem}[Proposition 3.8 of \cite{Stanley-f-vector}]\label{Thm-Stanley}
  Let $P$ be a simplicial poset of rank $r$ and let $\mathbf{k}[P]$ be its face ring. Then
  \[
    \operatorname{Hilb}(\mathbf{k}[P], t)=\frac{h_0 + h_1 t + \cdots + h_r t^r}{(1-t)^r}.
  \]
\end{Theorem}

\begin{Example}\label{ex:basic-simplicial-poset-Hilbert-series}
    Let us verify the previous theorem for our toy simplicial poset $P_1$. Its face ring is computed in Example \ref{ex:basic-simplicial-poset-ideal-computation} and its $\mathbf{f}$-vector and $\mathbf{h}$-vector are showed in Example \ref{ex:basic-simplicial-poset-f-h-computation}.
    
    \noindent
    By dirty hands computation or by using \texttt{Macaulay2} \cite{M2} we see that this Hilbert series simplifies to 
    \[
        \operatorname{Hilb}(\mathbf{k}[P_1], t)=\frac{1-t^2-t^4+t^6}{(1-t^2)^2(1-t^2)}=%\frac{(1+t^2)(1-t)^2(1+t)^2}{(1-t^2)^2(1-t)^2}=
        \frac{1+t^2}{(1-t)^2}
    \]
    and this is indeed the expected result.
%     \begin{scriptsize}
%     \begin{verbatim}
%         i1 : R = ZZ/101[a,b,p,q, Degrees=>{1,1,2,2}]
        
%         o1 = R
        
%         o1 : PolynomialRing
        
%         i2 : I = ideal(p*q, a*b-p-q)
        
%         o2 = ideal (p*q, a*b - p - q)
        
%         o2 : Ideal of R
        
%         i3 : hilbertSeries I
        
%                    2    4    6
%               1 - T  - T  + T
%         o3 = -----------------
%                    2 2       2
%              (1 - T ) (1 - T)
        
%         o3 : Expression of class Divide
%     \end{verbatim}
%     \end{scriptsize}
% this Hilbert series simplifies to 
%     \[
%         \frac{1-t^2-t^4+t^6}{(1-t^2)^2(1-t^2)}=\frac{(1+t^2)(1-t)^2(1+t)^2}{(1-t^2)^2(1-t)^2}=\frac{1+t^2}{(1-t)^2}
%     \]
%     and this is indeed the expected result.
\end{Example}

\begin{Example}\label{ex:basic-simplicial-Fig2-a-Hilbert-series}
    In the case of the face poset $P_2$, it is easy to verify what we have just stated.
    Indeed, in Example \ref{ex:basic-simplicial-Fig2-a}, the face ring and in Example \ref{ex:basic-simplicial-poset-f-h-computation} we computed the face ring, the $\mathbf{f}$-vector and $\mathbf{h}$-vector.
    %
    %\noindent
    It is trivial to observe the following:
    \[
        \operatorname{Hilb}\left(\frac{\mathbf{k}[x, y, z]}{(xz)}, t\right) = \frac{1-t^2}{(1-t)^3}=\frac{1+t}{(1-t)^2}.
    \]
\end{Example}

\subsection{Matroid over a ring $R$}
% %\section{Matroids, matroids over a ring and $\mathbb{Z}$-matroids.}
% %\paragraph{Matroids, Matroids over a ring and arithmetic Matroids.}
% A {\em matroid} $M$ is a simplicial complex $\mathcal{I}$ on a {\em ground set} $[n]=\{1,\dots,n\}$, such that
% \begin{equation*}%\label{eq-ind-set-exchange}
%  	A, B \in \mathcal{I}, \#A>\#B \Rightarrow \exists a\in A\setminus B: B\cup \{a\}\in \mathcal{I}.
% \end{equation*}
% %$$if A, B \in \mathcal{I}$ and $\#A>\#B$, then $\exists a\in A\setminus B: B\cup \{a\}\in \mathcal{I}.$$ 
% %
% The latter is called {\em independent set exchange property}; $\mathcal{I}$ is often called {\em independent sets} family and the facet of $\mathcal{I}$ are called {\em basis} of the matroid.

In \cite{FinkMoci}, Fink and Moci generalize the concept of matroid to matroid over a commutative ring $R$. In this section we give the general definition and, then, we deal with the case $R=\mathbb{Z}$.

\paragraph{$R$-matroids.}
Let $\binom{[n]}{[n]}$ be the set of all subsets of $[n]$ and let $R$--$\operatorname{mod}$ the category of finitely generated $R$-modules. 
\begin{defi}\label{def-matroid-over-a-ring}
A matroid over the ring $R$ is the function 
\[
	\mathcal{M}:\binom{[n]}{[n]}\rightarrow R\mbox{--}\operatorname{mod}
\]
such that for any subset $A$ of $[n]$ and any elements $b$ and $c$ of $[n]\setminus A$ there exist $x_{b,c}$ and $y_{b,c}$ in $\mathcal{M}(A)$, such that:
\begin{eqnarray*}
	\mathcal{M}(A\cup \{b\})&\simeq& \nicefrac{\mathcal{M}(A)}{(x_{b,c})}\\
	\mathcal{M}(A\cup \{c\})&\simeq& \nicefrac{\mathcal{M}(A)}{(y_{b,c})}\\
	\mathcal{M}(A\cup \{b,c\})&\simeq& \nicefrac{\mathcal{M}(A)}{(x_{b,c},y_{b,c})}	
\end{eqnarray*}
\end{defi}
\noindent
The choice of $\mathcal{M}$ is relevant only up to isomorphism. 
Moreover, we are going to assume that $\mathcal{M}$ is \emph{essential}, that is no non-trivial projective module is a direct summand of $\mathcal{M}([n])$. 

% \begin{Example}\label{ex:R-matroid-toric}
% 	Let $A=\mathbb{C}[x,y,x^{-1}, y^{-1}]$ and $[n]=[2]$. Set $\mathcal{M}(\emptyset)=A$, $\mathcal{M}(\{1\})=\nicefrac{A}{(xy-1)}$, $\mathcal{M}(\{2\})=\nicefrac{A}{(xy^{-1}-1)}$ and $\mathcal{M}(\{1,2\})=\nicefrac{A}{(xy-1, xy^{-1}-1)}$.
% \end{Example}

In Proposition 2.6 of \cite{FinkMoci}, Fink and Moci have shown that an essential matroid over a field $\mathbf{k}$ is a matroid in the classical case. 
For this reason, from now on, we are going to call these $\mathbf{k}$-matroids.

\paragraph{Realizable $\mathbb{Z}$-Matroids.}
For what regards this paper, we are going to set $R=\mathbb{Z}$. 
We define a corank function $\operatorname{cork}(A)$ of $\mathcal{M}$ as the corank function of the $\mathbb{Q}$-matroid $\mathcal{M}\otimes_{\mathbb{Z}} \mathbb{Q}$: $$\operatorname{cork}(A)=\operatorname{cork}_{\mathbb{Q}}\mathcal{M}\otimes_{\mathbb{Z}} \mathbb{Q}(A).$$
For any subset $A$ of $[n]$, $\mathcal{M}(A)$ is an abelian group,
\[
	\mathcal{M}(A)=\mathbb{Z}^{d(A)}\times G_A.
\]
where $G_A$ is the torsion part. We call $m(A)=\#G_A$ the \emph{multiplicity} of $A$. Easily, $\operatorname{cork}(A)=d(\emptyset)-d(A)$.

A $\mathbb{Z}$-matroid $\mathcal{M}$ on $[n]$ is realizable if there is a list of elements $z_1, \dots, z_n\in \mathcal{M}(\emptyset)$ such that $\mathcal{M}(A)=\nicefrac{\mathcal{M}(\emptyset)}{(z_i:i\in A)}$.

The definition (see next section) of the poset $\TM$ depends on the realization of the matroid $\mathcal{M}$. 
For this, when we need to use explicitly a realization of a matroid $\mathcal{M}$ we are going to talk about {\em realized matroid} $\mathcal{M}$. 

%When such $z_1, \dots, z_n$ belongs to $\mathbb{Z}^{d(\emptyset)}$ then we say that $\mathcal{M}$ is represented by a list of elements in $\mathbb{Z}^{d(\emptyset)}$.

\begin{Example}\label{ex:R-matroid-toric-Z-1}
	Let $R=\mathbb{Z}$ and $n=2$. Set $\mathcal{M}(\emptyset)=\mathbb{Z}^{2}$, $\mathcal{M}(\{1\})=\nicefrac{\mathbb{Z}^{2}}{(2,0)}$, $\mathcal{M}(\{2\})=\nicefrac{\mathbb{Z}^{2}}{(0,1)}$ and $\mathcal{M}(\{1,2\})=\nicefrac{\mathbb{Z}^{2}}{((2,0),(0,1))}$.
\end{Example}

\begin{Example}\label{ex:R-matroid-toric-Z-2}
	Let $A=\mathbb{Z}$ and $n=2$. Set $\mathcal{M}(\emptyset)=\mathbb{Z}^{2}$, $\mathcal{M}(\{1\})=\nicefrac{\mathbb{Z}^{2}}{(1,1)}$, $\mathcal{M}(\{2\})=\nicefrac{\mathbb{Z}^{2}}{(1,-1)}$ and $\mathcal{M}(\{1,2\})=\nicefrac{\mathbb{Z}^{2}}{((1,1),(1,-1))}$.
\end{Example}

\begin{Remark}
The definition and results in the next sections hold also for \emph{arithmetic matroids}.
We do not need any of the arithmetic matroids tools. 
In Section 6.1 of \cite{FinkMoci} it is shown that $\mathbb{Z}$-matroids with an extra \emph{molecule} condition are arithmetic matroids. 
(Not all arithmetic matroids can be seen as $\mathbb{Z}$-matroid.)
\end{Remark}

% \begin{Remark}
% The definition and results in the next sections hold for $\mathbb{Z}$-matroids and, so, naturally for \emph{arithmetic matroids}.
% %
% We do not need any of the arithmetic matroids tools, but some of the reader may be interested in this and so here we shortly recall a characterization presented by Fink and Moci in Section 6.1 of \cite{FinkMoci}.

% \noindent
% There is a technical definition needed to present the arithmetic matroids:
% Let $S\subseteq S'\subseteq [n]$. We say that the close interval $[S,S']$ is a \emph{molecule} if there exists a disjoint partition of $S'$ as $S\sqcup F \sqcup T$ where for each $A\in [S,S']$, $\operatorname{rk}(A)=\operatorname{rk}(S)+\#S\cap F$.

% \noindent
% Assume $\mathcal{M}$ is a matroid over $\mathbb{Z}$. 
% The triple ($\mathcal{M}$, $\operatorname{cork}$, $m$) is an \emph{arithmetic} matroid if for any molecule $[S, S']$ then
% \begin{equation*}
% 	\rho=(-1)^{\#T}\sum_{A\in [S,S']} (-1)^{\#S-\#A}m(A) \geq 0.
% \end{equation*}
% %Every matroid admits the trivial multiplicity function, everywhere equal to one and aside this case we refer to the couple $(\mathcal{M},m)$ as {\em arithmetic matroid}.
% \end{Remark}

\paragraph{Grothendieck--Tutte polynomial.}
In the rest of this section, we are going to define the Grothendieck--Tutte polynomial for matroid over $\mathbb{Z}$. 
%
%The more general definition (over a generic ring $R$) can be found in Section 7.1 of \cite{FinkMoci}.

Let $\operatorname{L}_0(\operatorname{Ab})$ be the Grothendieck type ring of abelian groups, that is the free group generated by the isomorphic classes $[G]$ for any finitely generated abelian group $G$. There is a ring multiplication given by $[G][G']=[G\times G']$.
%(We also consider product of such classes that $[G]\cdot[G'] as a symbol.)
%
This object is very useful; for instance it appears in \cite{Ekedahl-inv,EkedahlStack, Martino-eke, Martino-eke-intro}.
%check again.
%
In Section 7.1 of \cite{FinkMoci}, it is proved the Grothendieck--Tutte {\em class} is well defined to be the following element of $\operatorname{L}_0(\operatorname{Ab})\otimes \operatorname{L}_0(\operatorname{Ab})$:
\begin{equation*}%\label{eq-GT-poly-general}
	GT_{\mathcal{M}}=\sum_{A\subseteq E} [\mathcal{M}(A)][\mathcal{M}^*(E\setminus A)]%\in \operatorname{K}_0(A-\operatorname{mod})\otimes \operatorname{K}_0(A-\operatorname{mod})
\end{equation*}
where $\mathcal{M}^*$ is the matroid dual to $\mathcal{M}$ and $E$ is their common ground set. 
%
%(We use $E$ to avoid confusion with $[n]$.)
(We use $E$ to avoid confusion between $[n]$ and $[\mathcal{M}(-)]$.)
For a precise definition of the dual matroid $\mathcal{M}^*$ we refer %to  Section 4 of \cite{FinkMoci} or in 
Section 7 of \cite{Dadderio-Moci-Arith}. 

Let $G$ be a group and consider its class $[G]$ in $\operatorname{L}_0(\operatorname{Ab})$: since $G\simeq \mathbb{Z}^d\times G_t$, one has that $[G]=[\mathbb{Z}^d][G_t]\in \operatorname{L}_0(\operatorname{Ab})$. %, where $G_t$ is the torsion part of $G$. 
Now, fix the following evaluations for $\operatorname{L}_0(\operatorname{Ab})$: $v_x([G])=\#G_t (x-1)^{d}$ and, similarly, $v_y([G])=\#G_t (y-1)^{d}$), and consider the image of $GT_\mathcal{M}$ with respect to the map
\[
	v_x\otimes v_y :\operatorname{L}_0(\operatorname{Ab})\otimes_{\mathbb{Z}} \operatorname{L}_0(\operatorname{Ab}) \rightarrow \mathbb{Z}[x,y].
\]	
\noindent
The Grothendieck--Tutte polynomial for $\mathcal{M}$ is $T_{\mathcal{M}}(x,y)=(v_x\otimes v_y) (GT_{\mathcal{M}})$. Then,
\begin{equation}\label{eq-Tutte-classical}
	T_{\mathcal{M}}(x,y)=\sum_{A\subseteq [n]} m(A) (x-1)^{r-d(A)}(y-1)^{\#A - d(A)}.
\end{equation}
This polynomial was firstly introduced by Moci in \cite{Moci-A-tutte} and it is often called the arithmetic Tutte polynomial.
It easy to observe that 
\begin{equation}\label{eq-Tutte-classical-dual}
	T_{\mathcal{M}}(x,y)=T_{\mathcal{M}^*}(y,x).
\end{equation}

% Now, fix the following two evaluations for $\operatorname{K}_0(A-\operatorname{mod})$: $v_1([G])=\#G_t (x-1)^{d}$ and $v_2([G])=(y-1)^{d}$. Consider the image of $GT_\mathcal{M}$ with respect to the map
% \[
% 	v_1\otimes v_2 :\operatorname{K}_0(A-\operatorname{mod})\otimes \operatorname{K}_0(A-\operatorname{mod}) \rightarrow \mathbb{Z}[x,y].
% \]	
% \noindent
% Call $T_{\mathcal{M}}(x,y)=v_1\otimes v_2 (GT_{\mathcal{M}})$. Then,
% \begin{equation}\label{eq-Tutte-classical}
% T_{\mathcal{M}}(x,y)=\sum_{S\subseteq [n]} m(S) (x-1)^{r-\operatorname{rk}(S)}(y-1)^{\#A - \operatorname{rk}(S)}.
% \end{equation}

\noindent
%This is actually a specialization of the TG polynomial (see Section 7.1 of \cite{FinkMoci} for further details) but for what concerns this paper this is what we need.

% \begin{Example}\label{ex:matroids-tutte}
%     Dr. Lenz, please, compute the Tutte poly of the previous examples (ALL!)
% \end{Example}

\begin{Example}\label{ex:tutte-R-matroid-toric-Z-1}
	Let us compute $T_{\mathcal{M}}$ for the $\mathbb{Z}$-matroid given in Example \ref{ex:R-matroid-toric-Z-1}.
    In $(\ref{eq-Tutte-classical})$ the contribute  of the empty set is $(x-1)^2$; the contribution of the singleton $\{1\}$ is $2(x-1)$; the contribute of the singleton $\{2\}$ is $(x-1)$; finally, the contribute of the full ground set $[2]$ is $2$.
    Thus, $T_{\mathcal{M}}(x,y)=x^2+x$.
\end{Example}

\begin{Example}\label{ex:tutte-R-matroid-toric-Z-2}
	We compute $T_{\mathcal{M}}$ for the matroid in Example \ref{ex:R-matroid-toric-Z-2}.
    In $(\ref{eq-Tutte-classical})$ the contribution of the empty set is $(x-1)^2$; the contribute of each singleton is $(x-1)$; finally, the contribute of the full ground set $[2]$ is $2$.
    Thus, $T_{\mathcal{M}}(x,y)=x^2+1$.
\end{Example}

% \begin{Example}\label{ex:tutte-R-matroid-toric-Z-3}
% 	Let us compute $T_{\mathcal{M}}$ for the $\mathbb{Z}$-matroid given in Example \ref{ex:R-matroid-toric-Z-3}.
%     %
%     As before let us list the contribute in $(\ref{eq-Tutte-classical})$ of each subset:
% 	\begin{center}
% 	\begin{tabular}{cl}
%     $\emptyset$ 			& $(x-1)^2$\\
%     $\{1\},\{2\},\{3\}$ 	& $(x-1)$\\
%     $\{1,2\}$		 		& $2$\\
%     $\{1,3\}$		 		& $1$\\
%     $\{2,3\}$		 		& $1$\\
%     $\{1,2,3\}$		 		& $(y-1)$ 
%     \end{tabular}
%    	\end{center}
    
%     Thus, we get $T_{\mathcal{M}}(x,y)=x^2+x+y+1$.
% \end{Example}

As remarked in the introduction, the face ring and the Grothendieck--Tutte polynomial of a $\mathbf{k}$-matroid are related.

\begin{Theorem}\label{Thm-bjorner-hilbert-tutte}
    Let $\mathcal{M}$ be a $\mathbf{k}$-matroid of rank $r$.
    Let $\mathbf{k}[\mathcal{M}]$ be its face ring. 
    Then,
    \[
        \operatorname{Hilb}(\mathbf{k}[\mathcal{M}],t)=\frac{t^{r}}{(1-t)^{r}} T_{\mathcal{M}^*}(1, \nicefrac{1}{t}).
    \]
\end{Theorem}
\begin{proof}
Under the author knowledge, this result appears first in the above form in the Appendix (section A.3) by Bj\"orner in the work of De Concini and Procesi \cite{DC-P-Box}.
% %
% The proof follows in few steps:
% \begin{itemize}
% \item 
% \end{itemize}
\end{proof}

%\noindent
%\begin{Remark}
%	The naive observation that $T_{\mathcal{M}^*}(1, \nicefrac{1}{t})=T_{\mathcal{M}}(\nicefrac{1}{t},1)$ might mislead the reader. 
	%
%	The information about $\mathcal{M}$ should be read in $\mathbf{k}[\mathcal{M}^*]$ and this is why we stress it also in the previous theorem.
%\end{Remark}

% \begin{Example}
%     Compare the Hilbert series computed in Example \ref{ex:basic-simplicial-poset-Hilbert-series} and the Tutte polynomials computation in Example \ref{ex:matroids-tutte}.
% \end{Example}

% A matroid is said to be {\em realizable} over a field $\mathbf{k}$ if there exists $v_1, v_2, \dots, v_{\#E}$ in a $\mathbf{k}$-vector space $V$ such that $A=\{a_1, a_2\dots, a_k\}\in \mathcal{I}$ if and only if $v_{a_1}, v_{a_2}, \dots, v_{a_k}$ are linearly independent in $V$.

% \noindent
% The notion of realizability is slitly different in the arithmetic case. Indeed, $M$ is a {\em realizable arithmetic} matroid with $v_1, v_2, \dots, v_{\#E}\in \mathbb{Z}^d$ if for any $A \subseteq E$, $m(A)=\abs{\det(A)}$ providing that the vectors $v_{a_1}, v_{a_2}, \dots, v_{a_k}$ are a independent in $\mathbb{Z}^d$.
%
% \begin{Example}
%     Dr. Lenz, please, show here that the realizable matroids - above - where realizable. make argument about the non-realizability of the previous examples above.
% \end{Example}

The goal of the paper is to extend this result to realizable $\mathbb{Z}$-matroids.

\section{The poset of torsions}\label{sec-coset-of-poset}
The aim of this section is to define a new poset taking the role of the independent complex in the case of $\mathbf{k}$-matroids.

\textbf{All along this section we assume} that $A$ is a subset of $[n]$ and $b$ is in $[n]\setminus A$.
Let $b=c$ in Definition \ref{def-matroid-over-a-ring}: it is required the existence of a quotient homomorphism by $x_{b,b}\in \mathcal{M}(A)$: \begin{equation}\label{eq:canonicalquotient}
	\pi_{(A, b)}: \mathcal{M}(A) \rightarrow \mathcal{M}(A\cup \{b\}).
\end{equation}
Call $\pi_{(A, b)}$ the \emph{canonical projection} associated to $A$ and $b$.
While the homomorphism $\pi_{(A, b)}$ is unique, the choice of $x_{b,b}$ is not. 
In the case of realizable $\mathbb{Z}$-matroids, $x_{b,b}$ is unique and we denote it by $x_{b}$.

%Assume for simplicity that $\mathcal{M}(\emptyset)=\mathbb{Z}^d$, but the following argument works also if we drop this condition.
%
For any subset $A$, $\mathcal{M}(A)\simeq \mathbb{Z}^{d(A)} \times G_A$, where $d(A)$ is the rank of $\mathcal{M}(A)$ and $G_A$ is the torsion part of $\mathcal{M}(A)$.
Call $C_A=\operatorname{Hom}(G_A, \mathbb{C}^*)$, the dual group of $G_A$.

\begin{defi}\label{def-set-poset-of-torsion}
	We call $\TM$ the set of torsions of $\mathcal{M}$. 
	This is the set of all pairs $(A, l)$ with $d(\emptyset)-d(A)=\#A$ and $l\in C_A$.	
\end{defi}

\noindent
Inspired by Section 5 of \cite{Moci-A-tutte}, we are going to see such set as a bunch of tori with the right dimension and cardinality, prescribed by the $\mathbb{Z}$-matroid.
This is the reason that leads us to work with the dual group $C_A$ instead of $G_A$, even if they are isomorphic.
Moreover, consider $A$ and $A\cup \{b\}$ such that $d(\emptyset)-d(A)=\#A$ and $d(\emptyset)-d(A\cup \{b\})=\#A+1$. Then, the map $\pi_{(A, b)}$ restricted to $G_A$ is injective and its dual $\pi^{(A, b)}: C_{A\cup \{b\}}\twoheadrightarrow C_{A}$ is surjective.

%Call $\underline{l}$ the copy of $\mathbb{Z}^{d(A)}$ associated to $l\in C_A$, that is:
%\[
%\underline{l}=\{(m_1,\dots, m_{d(A)}, l): (m_1,\dots, m_{d(A)})\in \mathbb{Z}^{d(A)}\}\subseteq \mathcal{M}(A).%\subseteq \mathbb{Z}^{\operatorname{rk}(\mathcal{M}_{free}(A))}.
%\]
%
%\noindent
%Let $h$ be an element in the dual of the torsion group in $\mathcal{M}(A\cup \{b\})$.
%%
%Call $\overline{h}$ the coset of $h$ in $\mathcal{M}(A)$ as a quotient by $x_b$, i.e. $\mathcal{M}(A\cup \{b\})\simeq\nicefrac{\mathcal{M}(A)}{\left\langle x_b \right\rangle}$. 
%%
%\[
%\overline{h}=\{h+\left\langle x_b \right\rangle \}\subseteq \mathcal{M}(A).%\subseteq \mathbb{Z}^{\operatorname{rk}(\mathcal{M}_{free}(A))}.
%\]

\begin{defi}\label{def-covering-arithmetic}
	Let $(A\cup \{b\}, h)$ and $(A, l)$ be two elements of $\TM$.
	We say that $(A\cup \{b\}, h)$ covers $(A, l)$, and we write $(A\cup \{b\}, h)\rhd(A, l)$, if and only if $\pi^{(A, b)}(h)=l$.
\end{defi}

\begin{Example}\label{ex:R-matroid-toric-Z-1-poset-coset}
	Let us compute the poset of torsions of the matroid given in Example \ref{ex:R-matroid-toric-Z-1}. We show the poset in Figure \ref{fig:poset-coset-example}.a).
    Clearly there are six elements $(\emptyset, e)$, $(\{1\}, e)$, $(\{1\}, \zeta)$, $(\{2\}, e)$, $([2], e)$, and $([2], \zeta)$.
    
    \noindent
    Now observe that trivially $(\{1\}, e)$, $(\{1\}, \zeta)$, $(\{2\}, e)$ cover $(\emptyset, e)$ and $([2], x)$ covers $(\{2\}, e)$ because $C_{[2]}$ surjects to $C_{\{2\}}=\{e\}$.
    Moreover $([2], x)$ covers $(\{1\}, x)$, because $C_{[2]}\simeq C_{\{1\}}$, thus $\pi^{(\{1\}, 2)}(x)=x$. This also shows that $([2], x)$ does not cover $(\{1\}, y)$ if $x\neq y\in\nicefrac{\mathbb{Z}}{2\mathbb{Z}}$. 
\end{Example}

\begin{Example}\label{ex:R-matroid-toric-Z-2-poset-coset}
	The poset of torsions of the matroid given in Example \ref{ex:R-matroid-toric-Z-2} is actually the poset $P_1$ defined in Example \ref{ex:basic-simplicial-poset} and discussed all Section \ref{sec-basic-notion}.
    This poset is in Figure \ref{fig:poset-coset-example}.b).
    We leave to the reader to verify the covering relations, which are pretty straightforward.
\end{Example}

\begin{figure}[htb]
\begin{center}
\begin{tabular}{ccc}%{m{3cm}m{3cm}m{3cm}}
		\begin{tikzpicture}[scale=0.5, align=center]
            \node (1) at (-1.5,0) {$(12,e)$};
            \node (2) at (+1.5,0) {$(12,\zeta)$};
            
 			\node (a) at (-3, -2) {$(1,e)$};
            \node (b) at (0, -2) {$(1,\zeta)$};
 			\node (c) at (3, -2) {$(2,e)$};
 
            \node (0) at (0, -4) {$(\emptyset,e)$};

			\draw (1) -- (a);
            \draw (1) -- (c);
            \draw (2) -- (b);
            \draw (2) -- (c);

            \draw (a) -- (0);
            \draw (b) -- (0);
            \draw (c) -- (0);
        \end{tikzpicture}
&        \begin{tikzpicture}[scale=0.5, align=center]
            \node (1) at (0,0) {$(12,e)$};
            \node (2) at (3,0) {$(12,\zeta)$};
            \node (a) at (0, -2) {$(1,e)$};
            \node (b) at (3, -2) {$(2,e)$};
            \node (0) at (1.5, -4) {$\hat{0}$};
            \draw (1) -- (a);
            \draw (2) -- (a);
            \draw (1) -- (b);
            \draw (2) -- (b);
            \draw (a) -- (0);
            \draw (b) -- (0);
        \end{tikzpicture}    
&		\begin{tikzpicture}[scale=0.5, align=center]
            \node (1) at (-1.5,0) {$(12,e)$};
            \node (2) at (+1.5,0) {$(13,e)$};
            \node (3) at (+4.5,0) {$(23,e)$};
            \node (4) at (-4.5,0) {$(12,\zeta)$};

 			\node (a) at (-3, -2) {$(1,e)$};
            \node (b) at (0, -2) {$(2,e)$};
 			\node (c) at (3, -2) {$(3,e)$};
 
            \node (0) at (0, -4) {$\hat{0}$};

			\draw (1) -- (a);
            \draw (1) -- (b);
            \draw (2) -- (a);
 	        \draw (2) -- (c);
            \draw (3) -- (b);
 	        \draw (3) -- (c);
			\draw (4) -- (a);
            \draw (4) -- (b);
  
            \draw (a) -- (0);
            \draw (b) -- (0);
            \draw (c) -- (0);
        \end{tikzpicture}\\
  a)  & b) &   c)
\end{tabular}
\end{center}
\caption{a) The poset of torsions of the matroid given in Example \ref{ex:R-matroid-toric-Z-1}; b) The poset of torsions $P_1$ of the matroid given in Example \ref{ex:R-matroid-toric-Z-2} c) The poset of torsions of the matroid in Example \ref{ex:R-matroid-toric-Z-3}.}
   \label{fig:poset-coset-example}
\end{figure}
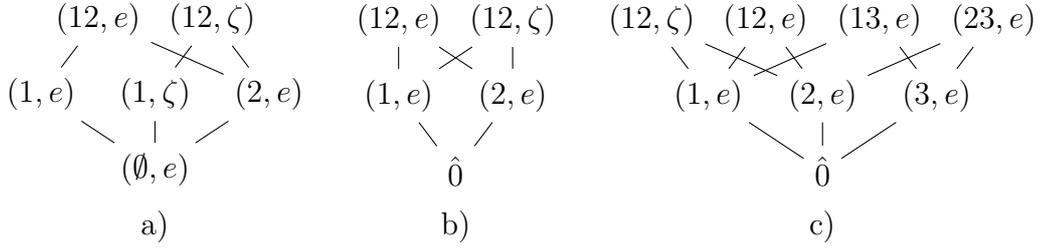

\section{$\TM$ is a union of simplicial posets}\label{sec-simplicial-poset}
In this section we are going to prove Theorem \ref{thm-simplicial-poset-intro}, that is $\TM$ is a a union of simplicial posets.
We start by proving two properties of $\TM$.

\begin{Proposition}\label{prop-the-covering-is-nice}
    Let $\mathcal{M}$ be a realizable matroid over $\mathbb{Z}$. 
    Let $(A\cup \{b\}, h)$, $(A, l_1)$, $(A, l_2)$ be in $\TM$. 
    If $(A\cup \{b\}, h)\rhd(A, l_1)$ and $(A\cup \{b\}, h)\rhd(A, l_2)$ then $l_1=l_2\in \mathcal{M}(A)$.
\end{Proposition}
\begin{proof}
	By Definition \ref{def-covering-arithmetic} if $(A\cup \{b\}, h)\rhd(A, l_1)$ and $(A\cup \{b\}, h)\rhd(A, l_2)$ then 
	$\pi^{(A, b)}(h)=l_1$ and $\pi^{(A, b)}(h)=l_2$, thus $l_1=l_2\in C_A$.
\end{proof}

\begin{Proposition}\label{prop-every-subset-is-there}
    Let $\mathcal{M}$ be a realizable matroid over $\mathbb{Z}$. 
    %
    %Let $A$ be an independent subset of $[n]$.
    %
    Let $(A\cup \{b\}, h)$ be in $\TM$. 
    %
    %If $(A\cup \{b\}, h)$ covers $(A, l_1)$ and if $(A\cup \{b\}, h)$ covers $(A, l_2)$, then $l_1=l_2$.
    Then, there exists $l\in G_A$ such that $(A\cup \{b\}, h)\rhd(A, l)$.
\end{Proposition}
\begin{proof}
	Consider the $\mathbb{Q}$-matroid $\mathcal{M}\otimes \mathbb{Q}$ and observe that $\mathcal{M}\otimes \mathbb{Q}(A)=\mathbb{Q}^{d(A)}$.
    Thus, $(A\cup \{b\}, h)\in\TM$ implies that $A\cup \{b\}$ belongs to the independent set complex of $\mathcal{M}\otimes \mathbb{Q}$. 
    This is a simplicial complex and so if $A\subset A\cup \{b\}$ then $A$ also belongs to the independent set complex and so, by definition, one has that $d(\emptyset)-d(A)=\#A$.
	
	\noindent
	Hence, remark that the map $\pi^{(A, b)}: C_{A\cup \{b\}} \rightarrow C_A$ is well defined and pick $l=\pi^{(A, b)}(h)$.
	Such $l\in C_A$ fulfills the statement request.
\end{proof}

\begin{Theorem}\label{thm-simplicial-poset}
    For every representable matroid $\mathcal{M}$ over $\mathbb{Z}$ with $\mathcal{M}(\emptyset)=\mathbb{Z}^{d}$, $\TM$ is a simplicial poset.
\end{Theorem}
\begin{proof}
	The element $(\emptyset, e)$ is the bottom element. 
    The only thing to check is that the interval $I=[(\emptyset, e), (A\cup\{b\}, h)]$ is boolean for every independent set $A\cup\{b\}$ and every $h\in C_{A\cup\{b\}}$.
    Using recursively Proposition \ref{prop-every-subset-is-there}, for every subset $E$ of $A\cup\{b\}$ there exists $l_e\in C_E$ such that $(E,l_e)$ belongs to $I$.
    Moreover, because of Proposition \ref{prop-the-covering-is-nice}, such subset $E$ appears only once in the interval $I$.
    Thus, $I$ is isomorphic as a poset to the boolean lattice $[\emptyset, A\cup\{b\}]$.     
\end{proof}

Many of the facts shown in Section 5 and 6 of \cite{Moci-A-tutte} can be proved as an application of the last theorem.
It is worth to mention, for instance that as a corollary of Theorem \ref{thm-simplicial-poset} one gets Lemma 6.1 of \cite{Moci-A-tutte}. 
\begin{Lemma}[Lemma 6.1 of \cite{Moci-A-tutte}]
	Let $\mathcal{M}$ be a realizable matroid over $\mathbb{Z}$. We call $E_{\mathcal{M}}(y)=\sum_{A\subseteq [n]} (y-1)^{\#A - \operatorname{cork}(A)}$ the polynomial of the external activity of $\mathcal{M}$.
    Denote by $C_0$ the pair $(A, l) \in \TM$ such that $d(A)=0$. 
    Finally call $M_{A}$ the restriction of the matroid $\mathcal{M}$ to $(A,l)$.
    
    Then, 
    $$T_{\mathcal{M}}(1,y)=\sum_{(A,l)\in C_0} E_{M_{A}}(y).$$ 
\end{Lemma}

\noindent
Indeed, the fact that each interval $[(\emptyset, e), (A\cup\{b\}, h)]$ is isomorphic to the boolean lattice $[\emptyset, A\cup\{b\}]$ implies that in the realizable arithmetic case, the toric arrangement associated looks locally as a hyperplane arrangement. 

By applying Theorem \ref{thm-simplicial-poset}, one can extend the main result to any realizable $\mathbb{Z}$-matroid.
To do this, we need the following technical definition.
Given an element $\sigma$ of a poset $P$ we denote the \emph{link} of $\sigma$ by $\operatorname{link}_P(\sigma)$:
	\[
    	\operatorname{link}_P(\sigma)=\{\tau\in P: \sigma \leq \tau\}\subseteq P.
    \]

\begin{TheoP}
	If $\mathcal{M}$ is a realizable $\mathbb{Z}$-matroid, then	$\TM$ is a disjoint union of $m(\emptyset)$ ($=\#G_\emptyset$) simplicial posets isomorphic to $\operatorname{link}_{\TM}(\emptyset, e)$.
\end{TheoP}
\begin{proof}
	If $\mathcal{M}(\emptyset)$ is a free group, we have already proved that the statement is true in Theorem \ref{thm-simplicial-poset}.
    %
    %\noindent
	If $\mathcal{M}(\emptyset)$ is not free, pick $c\in C_\emptyset$. Each pair $(\emptyset, c)$ is minimal in $\TM$.
    Moreover, there is a natural poset isomorphism from the elements of $\operatorname{link}_{\TM}(\emptyset, e)$ to the elements of $\operatorname{link}_{\TM}(\emptyset, c)$. 
    The isomorphism sends $(E, l)\in \operatorname{link}_{\TM}(\emptyset, e)$ to $(E, cl)\in \operatorname{link}_{\TM}(\emptyset, c)$.
    
    \noindent
    Finally, define for every $A\subseteq[n]$,  $\mathcal{M}'(A)=\nicefrac{\mathcal{M}(A)}{G_\emptyset}$. 
    This is a realizable $\mathbb{Z}$-matroid and $\mathcal{M}'(\emptyset)$ is free. Moreover, $\operatorname{Gr}(\mathcal{M}')=\operatorname{link}_{\TM}(\emptyset, e)$. 
    %
    %Then the statement follows by using Theorem \ref{thm-simplicial-poset} to the matroid $\mathcal{M}'$.
\end{proof}

\section{The Hilbert series of the face module}\label{sec-main-result}
In this section we show that the face module and the Grothendieck--Tutte polynomial of a realizable $\mathbb{Z}$-matroid are related as in the classical case.  
%
%From now on, we are going to assume that $m(\emptyset)=1$ or equivalently that $\mathcal{M}(\emptyset)$ is free. 

% \begin{Theorem} 
% 	Let $\mathcal{M}$ be a matroid over $\mathbb{Z}$ of rank $r$ with ground set $[n]$.
%    	%
%     Let $\TM$ be the p.o.set of cosets of $\mathcal{M}$.
    
%     Then, $\TM$ is shellable.
% \end{Theorem}
% \begin{proof}
% 	ciao
% \end{proof}

Recall that the face ring of a simplicial poset $P$ has been defined in Section \ref{sec-basic-notion} and it is denoted by $\mathbf{k}[P]$.
Theorem \ref{thm-simplicial-poset} shows that if $\mathcal{M}$ is representable matroid over $\mathbb{Z}$ with $\mathcal{M}(\emptyset)=\mathbb{Z}^{d}$, then $\TM$ is a simplicial poset.
Therefore, we might define the face ring of such matroid $\mathcal{M}$ as
\[
\mathbf{k}[\mathcal{M}]=\mathbf{k}[\operatorname{Gr}(\mathcal{M})].
\]

In the general realizable case $\mathcal{M}(\emptyset)=\mathbb{Z}^{d(\emptyset)}\times G_{\emptyset}$, Theorem \ref{thm-simplicial-poset-intro} ensures that $\TM$ is a union of $m(\emptyset)=\#G_\emptyset$ a simplicial posets.
The right algebraic structure is not anymore a ring, but a module.
Gluing all this fact together, the reader can make sense of the following definition.

\begin{defi}
The face module $\mathbf{k}[\mathcal{M}]$ of $\mathcal{M}$ is 
\[
	\mathbf{k}[\mathcal{M}]=\mathbf{k}[\operatorname{Gr}(\mathcal{M}')]^{m(\emptyset)},
\]
%$\mathbf{k}[\mathcal{M}]=\mathbf{k}[\operatorname{Gr}(\mathcal{M}')]^{m(\emptyset)}$, 
where $\mathcal{M}'$ is the matroid defined for every $A\subseteq [n]$ by $\mathcal{M}'(A)=\nicefrac{\mathcal{M}(A)}{G_\emptyset}$.
\end{defi}

\noindent
Note that $\mathbf{k}[\mathcal{M}]$ is a free module over the ring $\mathbf{k}[\operatorname{link}_{\TM}(\emptyset, e)]$.
In other words, $\mathbf{k}[\mathcal{M}]=\mathbf{k}[\operatorname{link}_{\TM}(\emptyset, e)]^{m(\emptyset)}$.
If $\mathcal{M}(\emptyset)$ is free than the face module has a ring structure, i.e. $\mathbf{k}[\mathcal{M}]=\mathbf{k}[\TM]$.
Finally, recall that the dual of a realizable $\mathbb{Z}$-matroid is still realizable (see Section 2 of \cite{Dadderio-Moci-Arith}).

\begin{TheoP}%\label{thm-GT-poly}
	If $\mathcal{M}$ is a realizable $\mathbb{Z}$-matroid of rank $r$, then
	\[
	\operatorname{Hilb}(\mathbf{k}[\mathcal{M}],t)=\frac{t^{r}}{(1-t)^{r}} T_{\mathcal{M}^*}(1, \nicefrac{1}{t}).
	\]
\end{TheoP}
\begin{proof}
	For the additivity property of the Hilbert series, it is enough to show that the theorem is true in the case $m(\emptyset)=1$.
    $\TM$ is a simplicial poset because of Theorem \ref{thm-simplicial-poset}. One defines its $\mathbf{h}$-vector as 
    $$\sum_{i=0}^{r} f_{i-1}(\TM) (t-1)^{r-i}=\sum_{i=0}^r h_{i}(\TM) t^{r-i}.$$
    We observe that 
    $$f_{i-1}(\TM)=\sum_{\# A=i}m(A),$$ 
    where $m(A)$ is the order of the torsion part of $\mathcal{M}(A)$.
    Hence
    \[
    	\sum_{i=0}^r h_{i}(\TM)t^{r-i}=\sum_{A\in[n] }m(A)(t-1)^{r-d(A)}=T_{\mathcal{M}}(t,1).
    \]
    Therefore, 
    $t^{r}T_{\mathcal{M}}(\nicefrac{1}{t},1)=\sum_{i=0}^r h_{i}(\TM)t^{i}$.
    We now apply Theorem \ref{Thm-Stanley} together with (\ref{eq-Tutte-classical-dual}) to get the result.
\end{proof}
% \begin{proof}
% 	Assume for now that $\mathcal{M}(\emptyset)=\mathbb{Z}^d$.
% 	%
% 	Because of Theorem \ref{thm-simplicial-poset}, the $\TM$ is a simplicial poset and because of Theorem \ref{Thm-Stanley} one gets:
% 	\[
%     \operatorname{Hilb}(\mathbf{k}[\mathcal{M}], t)=\frac{h_0 + h_1 t + \cdots + h_r t^r}{(1-t)^r}.
%   	\]
%     \noindent
%     Remark that $\sum_{i=0}^{r} f_{i-1}(t-1)^{r-i}=\sum_{i=0}^r h_{i}t^{r-i}$ and observe that $f_{i-1}=\sum_{\# S=i}m(S)$.
%     %
%     Thus one check that 
%     \[
%     	\sum_{i=0}^r h_{i}t^{r-i}=\sum_{S\in[n] }m(S)(t-1)^{r-\operatorname{rk}(A)}=T_{\mathcal{M}}(1,y).
%     \]
% \end{proof}   
    
% \begin{Corollary}\label{cor-GT-poly-torsion}
% 	Let $\mathcal{M}$ be a $\mathbb{Z}$-matroid of rank $r$ with ground set $[n]$. 
%     %
%     %Assume $\mathcal{M}(\emptyset)$ has no torsion.
%    	%
%     Let $\TM$ be the p.o.set of cosets of $\mathcal{M}$ and denote its face poset by $\mathbf{k}[\mathcal{M}]=\mathbf{k}[\TM]$.
%     %
%   	Then 
%   	\[
%     	\operatorname{Hilb}(\mathbf{k}[\mathcal{M}],t)=\frac{t^{n-r}}{(1-t)^{n-r}} T_{\mathcal{M}}(1, \nicefrac{1}{t}).
%     \]
% \end{Corollary}

% \begin{proof}
%     If $\mathcal{M}(\emptyset)$ is not free, by Corollary \ref{cor-main-result}, $\TM$ is a union of disjoint isomorphic simplicial poset and we apply Theorem \ref{thm-GT-poly}. 
% \end{proof}

\begin{Remark}\label{remark-conjecture}
The proof of the above theorem works for every simplicial partial order of the set in Definition \ref{def-set-poset-of-torsion}.
We conjecture Theorem \ref{thm-main-result} is true for every matroid over $\mathbb{Z}$ and the only obstacle to this result is hidden in the nature of the {\em canonical projections}. 
Indeed, for a non realizable $\mathbb{Z}$-matroid, it is not clear if there is a unique simplicial order of the set in Definition \ref{def-set-poset-of-torsion}, that respects Definition \ref{def-covering-arithmetic}.
\end{Remark}

Along all the paper we have played with two toy examples: $\mathcal{M}_1$ defined in Example \ref{ex:R-matroid-toric-Z-1} and $\mathcal{M}_2$ defined in Example \ref{ex:R-matroid-toric-Z-2}. In Table \ref{table:examples} we summarize where to find the calculations related: the computations of the poset, the Tutte polynomial, the face ring, etc. 

\begin{table}[htb]
\begin{center}
\scalebox{0.9}{
\begin{tabular}{cccccccc}
 & $\mathcal{M}$ & $T_{\mathcal{M}}$ & $\TM$  & $\mathbf{h}$, $\mathbf{f}$ & $\mathbf{k}[\mathcal{M}]$ & Hilb & Fig.\\
$\mathcal{M}_1$ & Ex. \ref{ex:R-matroid-toric-Z-1} & Ex. \ref{ex:tutte-R-matroid-toric-Z-1} & Ex. \ref{ex:R-matroid-toric-Z-1-poset-coset}  & Ex. \ref{ex:basic-simplicial-Fig2-a-f-h-computation} & Ex. \ref{ex:basic-simplicial-Fig2-a} & Ex. \ref{ex:basic-simplicial-Fig2-a-Hilbert-series}& Fig. \ref{fig:poset-coset-example}.a)\\
$\mathcal{M}_2$ & Ex. \ref{ex:R-matroid-toric-Z-2} & Ex. \ref{ex:tutte-R-matroid-toric-Z-2} & Ex. \ref{ex:R-matroid-toric-Z-1-poset-coset}  & Ex. \ref{ex:basic-simplicial-poset-f-h-computation} & Ex. \ref{ex:basic-simplicial-poset-ideal-computation} & Ex. \ref{ex:basic-simplicial-poset-Hilbert-series}& Fig. \ref{fig:poset-coset-example}.b)
\end{tabular}}
\end{center}
\caption{The computations of the toy examples $\mathcal{M}_1$ and $\mathcal{M}_2$.}\label{table:examples}
\end{table}

We now provide a more substantial example where to verify the just proved Theorem \ref{thm-main-result}. 

\begin{Example}\label{ex:R-matroid-toric-Z-3}
	Let $n=3$ and we define $\mathcal{M}$ as follows:
    
    \begin{center}
    \scalebox{0.75}
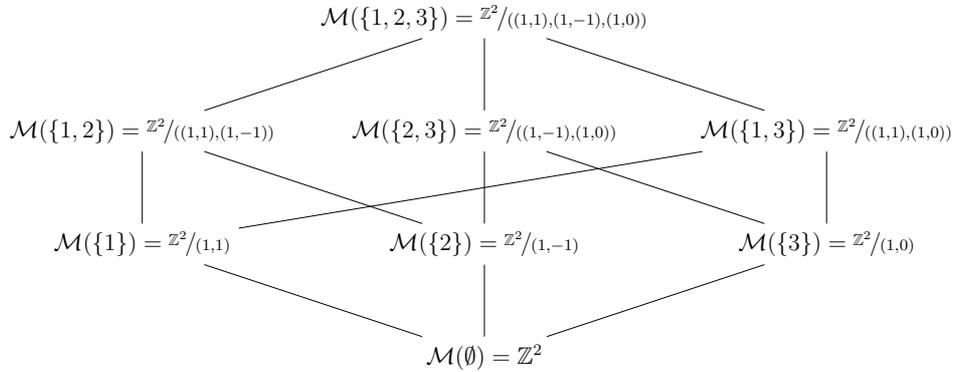
{%
   	\begin{tikzpicture}[scale=1, align=center]
            \node (1) at (-6,0) {$\mathcal{M}(\{1,2\})=\nicefrac{\mathbb{Z}^{2}}{((1,1),(1,-1))}$};
            \node (2) at (0,0) {$\mathcal{M}(\{2,3\})=\nicefrac{\mathbb{Z}^{2}}{((1,-1),(1,0))}$};
            \node (3) at (+6,0) {$\mathcal{M}(\{1,3\})=\nicefrac{\mathbb{Z}^{2}}{((1,1),(1,0))}$};
            
            \node (123) at (0,2) {$\mathcal{M}(\{1,2,3\})=\nicefrac{\mathbb{Z}^{2}}{((1,1),(1,-1),(1,0))}$};

 			\node (a) at (6, -2) {$\mathcal{M}(\{3\})=\nicefrac{\mathbb{Z}^{2}}{(1,0)}$};
            \node (b) at (0, -2) {$\mathcal{M}(\{2\})=\nicefrac{\mathbb{Z}^{2}}{(1,-1)}$};
 			\node (c) at (-6, -2) {$\mathcal{M}(\{1\})=\nicefrac{\mathbb{Z}^{2}}{(1,1)}$};
 
            \node (0) at (0, -4) {$\mathcal{M}(\emptyset)=\mathbb{Z}^{2}$};

			\draw (a) -- (2);
            \draw (a) -- (3);
            \draw (b) -- (1);
 	        \draw (b) -- (2);
            \draw (c) -- (1);
 	        \draw (c) -- (3);

			\draw (1) -- (123);
            \draw (2) -- (123);
			\draw (3) -- (123);

            \draw (a) -- (0);
            \draw (b) -- (0);
            \draw (c) -- (0);
        \end{tikzpicture}
   	}\end{center}
    
%    \renewcommand\arraystretch{1} 
%     Set $\mathcal{M}(\emptyset)=\mathbb{Z}^{2}$, $\mathcal{M}(\{1\})=\nicefrac{\mathbb{Z}^{2}}{(1,1)}$, $\mathcal{M}(\{2\})=\nicefrac{\mathbb{Z}^{2}}{(1,-1)}$, $\mathcal{M}(\{3\})=\nicefrac{\mathbb{Z}^{2}}{(1,0)}$, $\mathcal{M}(\{1,2\})=\nicefrac{\mathbb{Z}^{2}}{((1,1),(1,-1))}$, $\mathcal{M}(\{1,3\})=\nicefrac{\mathbb{Z}^{2}}{((1,1),(1,0))}$, $\mathcal{M}(\{2,3\})=\nicefrac{\mathbb{Z}^{2}}{((1,-1),(1,0))}$ and $\mathcal{M}(\{1,2,3\})=\nicefrac{\mathbb{Z}^{2}}{((1,1),(1,-1),(1,0))}$.
% \end{Example}

% \begin{Example}\label{ex:tutte-R-matroid-toric-Z-3}
	
    \noindent
	Let us compute $T_{\mathcal{M}}$. % for the $\mathbb{Z}$-matroid given in Example \ref{ex:R-matroid-toric-Z-3}.
    We list the contribute in $(\ref{eq-Tutte-classical})$ for each subset:
	\begin{center}
	\begin{tabular}{cl}
    $\emptyset$ 			& $(x-1)^2$\\
    $\{1\},\{2\},\{3\}$ 	& $(x-1)$\\
    $\{1,2\}$		 		& $2$\\
    $\{1,3\}$		 		& $1$\\
    $\{2,3\}$		 		& $1$\\
    $\{1,2,3\}$		 		& $(y-1)$ 
    \end{tabular}
   	\end{center}
    
    \noindent
    Thus, we get $T_{\mathcal{M}}(x,y)=x^2+x+y+1$.
% \end{Example}

% \begin{Example}\label{ex:R-matroid-toric-Z-3-poset-coset}
%	Let us focus on the matroid with ground set $[3]$ presented in Example \ref{ex:R-matroid-toric-Z-3}. 
    %
%    In this example the condition $\operatorname{cork}(S)=\#S$ plays a role.
    
    Now we construct $\TM$.
    We start by observing that the matroid in Example \ref{ex:R-matroid-toric-Z-2} is a submatroid of $\mathcal{M}$. We have already studied this submatroid and therefore we do not need to explain the covering relation among $(\emptyset, e)$, $(\{1\}, e)$, $(\{2\}, e)$, $(\{1,2\}, e)$, and $(\{1,2\}, \zeta)$.
    
    \noindent
    In $\TM$, we also find $(\{3\}, e)$, $(\{1,3\}, e)$, and $(\{2,3\}, e)$. 
    We remark that the subset $[3]$ does not appear in the poset, because $2=\operatorname{cork}([3]) \neq \#[3]=3$.
    %
    %\noindent
    Thus, it remains to study which elements are covered by the subsets containing $3$.
    Readily, $(\{3\}, e)$ covers $(\emptyset, e)$.
    
    \noindent
    Since $C_{\{2,3\}}$, $C_{\{1,3\}}$ are trivial groups, then $(\{1,3\}, e)$ covers $(\{1\}, e)$ and $(\{3\}, e)$, and similarly $(\{2,3\}, e)$ coves $(\{2\}, e)$ and $(\{3\}, e)$.
    Figure \ref{fig:poset-coset-example}.c) shows $\TM$.
    
   	Using Macaulay2 \cite{M2}, we compute the Hilbert series of the face ring:
    \[
    	\operatorname{Hilb}(\mathbf{k}[\mathcal{M}],t)=\frac{1+t+2t^2}{(1-t)^2}.
    \]
    
    Let us focus on the dual matroid; one can easily compute that 
    \begin{center}
    \scalebox{0.75}
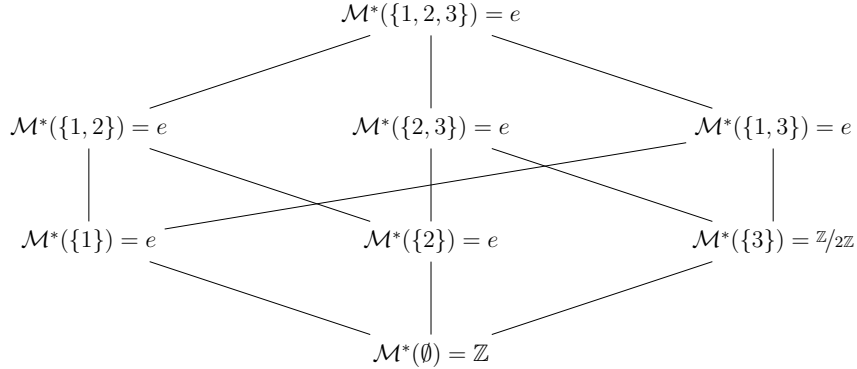
{%
    \begin{tikzpicture}[scale=1, align=center]
            \node (1) at (-6,0) {$\mathcal{M}^*(\{1,2\})=e$};
            \node (2) at (0,0) {$\mathcal{M}^*(\{2,3\})=e$};
            \node (3) at (+6,0) {$\mathcal{M}^*(\{1,3\})=e$};
            
            \node (123) at (0,2) {$\mathcal{M}^*(\{1,2,3\})=e$};

 			\node (a) at (6, -2) {$\mathcal{M}^*(\{3\})=\nicefrac{\mathbb{Z}}{2\mathbb{Z}}$};
            \node (b) at (0, -2) {$\mathcal{M}^*(\{2\})=e$};
 			\node (c) at (-6, -2) {$\mathcal{M}^*(\{1\})=e$};
 
            \node (0) at (0, -4) {$\mathcal{M}^*(\emptyset)=\mathbb{Z}$};

			\draw (a) -- (2);
            \draw (a) -- (3);
            \draw (b) -- (1);
 	        \draw (b) -- (2);
            \draw (c) -- (1);
 	        \draw (c) -- (3);

			\draw (1) -- (123);
            \draw (2) -- (123);
			\draw (3) -- (123);

            \draw (a) -- (0);
            \draw (b) -- (0);
            \draw (c) -- (0);
        \end{tikzpicture}}
   	\end{center}
    
    \noindent
    By duality $T_{\mathcal{M}^*}(x,y)=y^2+y+x+1$ and by trivial computation, 
  	\[
    	\operatorname{Hilb}(\mathbf{k}[\mathcal{M}^*],t)=\frac{1+3t}{(1-t)}.
    \]
    
    %\noindent
    %We leave to the reader to verify Theorem \ref{thm-main-result}.
\end{Example}

%There is still to prove that 
%
%
%\begin{Thm}[cite someone]
%  $\hilb(\mathbf{k}[M],t)=\frac{t^{\#E-d} \aritutte_M(\nicefrac{1}{t}, 1)}{(1-t)^{\#E-d}}$.
%\end{Thm}

% \section{The M\'alaga variety}

% \begin{defi}
%   Let $X\subseteq \Z^d$ be a list of vectors. The M\'alaga variety $\malaga_X$ is defined as
  
%   to be filled in by Ivan
% \end{defi}

% \begin{Theorem}
%     Let $X\subseteq \Z^d$ be a list of vectors.
%     Then the arithmetic Stanley-Reisner ring is isomorphic to the cohomology of  $\malaga_X$.
% \end{Theorem}
% \begin{proof}
%   to be filled in by Ivan
% \end{proof}

% \begin{Remark}
% This paper could be related to the M\'alaga variety:
% On the integral cohomology of toric varieties,
% Jin Hong Kim,
% \url{http://www.worldscientific.com/doi/10.1142/S0219498816500328}
% \end{Remark}

% \section{Arithmetic $h$-vectors}

%  One should note that Aaron Dall, a guy from Barcelona (not M\'alaga!) has studied arithmetic $h$-vectors in chapter three of his thesis.
%  So, apart from Lenz and Lenz--Martino, there is now a third reference on this topic. We should read his thesis at some point.

\bibliographystyle{alpha}
\bibliography{biblio}

\begin{thebibliography}{Moc12b}

\bibitem[CD17]{Delucchi-Callegaro-toric}
Filippo Callegaro and Emanuele Delucchi.
\newblock The integer cohomology algebra of toric arrangements.
\newblock {\em Advances in Mathematics}, 313:746 -- 802, 2017.

\bibitem[CG17]{DeC-Gaiffi-Projective-Wonderful}
Corrado~De Concini and Giovanni Gaiffi.
\newblock Projective wonderful models for toric arrangements.
\newblock {\em Advances in Mathematics}, 2017.

\bibitem[DCP08]{DC-P-Box}
C.~De~Concini and C.~Procesi.
\newblock Hyperplane arrangements and box splines.
\newblock {\em Michigan Math. J.}, 57:201--225, 2008.
\newblock With an appendix by A. Bj\"orner, Special volume in honor of Melvin
  Hochster.

\bibitem[DM13]{Dadderio-Moci-Arith}
Michele D'Adderio and Luca Moci.
\newblock Arithmetic matroids, the {T}utte polynomial and toric arrangements.
\newblock {\em Adv. Math.}, 232:335--367, 2013.

\bibitem[Eke09a]{Ekedahl-inv}
Torsten Ekedahl.
\newblock A geometric invariant of a finite group.
\newblock arXiv:0903.3148v1, 2009.

\bibitem[Eke09b]{EkedahlStack}
Torsten Ekedahl.
\newblock The {G}rothendieck group of algebraic stacks.
\newblock arXiv:0903.3143v2, 2009.

\bibitem[FM16]{FinkMoci}
Alex Fink and Luca Moci.
\newblock Matroids over a ring.
\newblock {\em J. Eur. Math. Soc. (JEMS)}, 18(4):681--731, 2016.

\bibitem[GS]{M2}
Daniel~R. Grayson and Michael~E. Stillman.
\newblock Macaulay2, a software system for research in algebraic geometry.
\newblock Available at \url{http://www.math.uiuc.edu/Macaulay2/}.

\bibitem[Mar16]{Martino-eke}
Ivan Martino.
\newblock The {E}kedahl invariants for finite groups.
\newblock {\em J. Pure Appl. Algebra}, 220(4):1294--1309, 2016.

\bibitem[Mar17]{Martino-eke-intro}
Ivan Martino.
\newblock Introduction to the {E}kedahl {I}nvariants.
\newblock {\em MATH. SCAND.}, 120:211–224a, 2017.

\bibitem[Moc12a]{Moci-A-tutte}
Luca Moci.
\newblock A {T}utte polynomial for toric arrangements.
\newblock {\em Trans. Amer. Math. Soc.}, 364(2):1067--1088, 2012.

\bibitem[Moc12b]{Moci-Wonderful-Model}
Luca Moci.
\newblock Wonderful models for toric arrangements.
\newblock {\em Int. Math. Res. Not. IMRN}, (1):213--238, 2012.

\bibitem[Sta91]{Stanley-f-vector}
Richard~P. Stanley.
\newblock {$f$}-vectors and {$h$}-vectors of simplicial posets.
\newblock {\em J. Pure Appl. Algebra}, 71(2-3):319--331, 1991.

\end{thebibliography}

%signature
\vspace{0.5cm}
 
\noindent
 {\scshape Ivan Martino}\\
 {\scshape Department of Mathematics, Northeastern University,\\ Boston, MA 02115, USA}.\\
 {\itshape E-mail address}: \texttt{i.martino@northeastern.edu}
\end{document}